\documentclass[10pt]{amsart}
\usepackage{graphicx} 
\usepackage{amsmath, amsthm, amssymb}
\usepackage[usenames,dvipsnames,svgnames,table]{xcolor}
\usepackage[margin=1.3in]{geometry}
\usepackage{esint}

\newtheorem{theorem}{Theorem}[section]
\newtheorem{proposition}[theorem]{Proposition}

\newtheorem{lemma}[theorem]{Lemma}

\newtheorem{corollary}[theorem]{Corollary}

\newcommand{\R}{\mathbb R}

\newcommand{\eps}{\varepsilon}

\newcommand{\dd}{\, \mathrm{d}}

\DeclareMathOperator{\dist}{dist}

\numberwithin{equation}{section}

\title[Regularity for nonlocal Bernoulli problems]{Regularity and nondegeneracy for nonlocal Bernoulli problems with variable kernels}
\author[Stanley Snelson]{Stanley Snelson}
\address{Department of Mathematics and Systems Engineering, 150 W. University Blvd, Florida Institute of Technology, Melbourne, FL 32901}
\email{ssnelson@fit.edu}

\author[Eduardo V. Teixeira]{Eduardo V. Teixeira}
\address{Department of Mathematics, University of Central Florida, 4393 Andromeda Loop N, Orlando, FL 32816}
\email{eduardo.teixeira@ucf.edu}

\thanks{SS was partially supported by NSF grant DMS-2213407 and a Collaboration Grant from the Simons Foundation, Award \#855061.}

\begin{document}

\maketitle

\begin{abstract}
    We consider a generalization of the Bernoulli free boundary problem where the underlying differential operator is a nonlocal, non-translation-invariant elliptic operator of order $2s\in (0,2)$. Because of the lack of translation invariance, the Caffarelli-Silvestre extension is unavailable, and we must work with the nonlocal problem directly instead of transforming to a thin free boundary problem. We prove global H\"older continuity of minimizers for both the one- and two-phase problems. Next, for the one-phase problem, we show H\"older continuity at the free boundary with the optimal exponent $s$. We also prove matching nondegeneracy estimates. A key novelty of our work is that all our findings hold without requiring any regularity assumptions on the kernel of the nonlocal operator. This characteristic makes them crucial in the development of a universal regularity theory for nonlocal free boundary problems.

\end{abstract}

\section{Introduction}

We consider a free boundary problem in variational form, that is described as follows: Let $K(x,y)$ be a measurable kernel satisfying the ellipticity condition
\begin{equation}\label{e:ellipticity}
(1-s) \frac \lambda {|x-y|^{d+2s}} \leq K(x,y) \leq (1-s) \frac \Lambda {|x-y|^{d+2s}}, \quad x,y \in \R^d,
\end{equation}
for some $s\in (0,1)$ and $\Lambda > \lambda >0$, as well as the symmetry condition
\begin{equation}\label{e:symmetry}
K(x,y) = K(y,x).
\end{equation}
For a bounded domain $\Omega \subset \R^d$, define the functional
\begin{equation}\label{e:JK}
\mathcal J_{K}(u) := \iint_{\R^{2d}\setminus (\Omega^c)^2}  K(x,y) |u(y) - u(x)|^2 \dd y \dd x + |\{u>0\}\cap \Omega|,
\end{equation}
where $|\cdot|$ is Lebesgue measure on $\R^d$. For a given function $g \in L^2(\R^d\setminus \Omega)$, we want to minimize $\mathcal J_K$ over the space
\[
\mathcal H^s_g(\Omega) := \{ w \in L^2(\R^d) \cap H^s(\Omega) : w = g \text{ a.e. in } \R^d\setminus \Omega\}.
\]
The integral in \eqref{e:JK} is defined over $\R^{2d}\setminus (\Omega^c)^2$ because variations in the space $\mathcal H^s_g(\Omega)$ do not affect the integral over $(\Omega^c)^2$.


We are interested in both the {\it one-phase problem}, where $g \geq 0$ and the minimizer $u$ can be chosen nonnegative, and the {\it two-phase problem}, where $g$ and $u$ may change sign.


For any kernel $K$, we denote the corresponding nonlocal operator by
\[
\mathcal L_Ku = \int_{\R^d} K(x,y) [u(y) - u(x)] \dd y.
\]

The formal Euler-Lagrange equation of our variational problem is the overdetermined nonlocal equation
\begin{equation}\label{e:bernoulli}
\begin{split}
\mathcal L_K u &= 0, \quad x\in \Omega\cap \{u>0\},\\
u(x) &= g(x), \quad x\in \R^d\setminus \Omega,\\
\lim_{\Omega \cap\{u>0\} \ni \xi \to x} \frac{u(\xi) - u(x)}{((\xi-x)\cdot\nu(x))^s} &= A(x), \quad x\in \partial\{u>0\}\cap \Omega,
\end{split}
\end{equation}
for some $A(x)$ depending on the kernel $K$. This generalizes the classical Bernoulli problem, which consists of finding a function $u$ that is harmonic in $\{u>0\}$ and whose normal derivative from inside $\{u>0\}$ equals 1 at boundary points of $\{u>0\}$. 


There is a broad interest in Bernoulli-type problems for the fractional Laplacian, which corresponds to $K(x,y) = c|x-y|^{-d-2s}$. (See Section \ref{s:related} below for a partial bibliography.) The nonlocality of the underlying operator makes it difficult to apply the standard techniques, as variations in a small region affect the solution everywhere. 
For this reason, the most common approach uses the Caffarelli-Silvestre extension to relate the fractional problem to a thin cavitation problem in $\R^{d+1}$. This corresponds to minimizing
\[
\mathcal J_{\rm thin}(u) = \int_{\R^{d+1}} |x_{n+1}|^{1-2s}|\nabla u|^2 \dd x + \mathcal L^d(\{u>0\}\cap \Omega \cap \{x_{n+1}=0\}),
\]
where $u$ is defined on $\R^{d+1}$. A solution of the Bernoulli problem for the $d$-dimensional fractional Laplacian $(-\Delta)^s$ yields a minimizer of $\mathcal J_{\rm thin}$ via the Poisson formula. 
However, this extension technique is unavailable for general kernels $K(x,y)$. Our goal in this article is to develop a new approach that can handle any kernel satisfying \eqref{e:ellipticity} and \eqref{e:symmetry}. In particular, we do not assume any continuity of our kernels with respect to $x$ and $y$.

\subsection{Main results}

Our results are stated in terms of the tail of a minimizer $u$. Borrowing the notation of \cite{DKP2016}, we define, for any $R>0$ and $x_0 \in \R^d$,
\begin{equation}\label{e:tail}
    \mbox{Tail}(u;x_0,R) = R^{2s} \int_{\R^d\setminus B_{R}(x_0)} \frac {|u(x)|} {|x-x_0|^{d+2s}} \dd x.
\end{equation}

Our first main result is global H\"older continuity for minimizers:

\begin{theorem}\label{t:reg}
With $\Omega$ and $g$ as above, there exist $\alpha \in (0,1)$, $C>0$, and $r_0 >0$, depending only on $d$, $s$, $\lambda$, and $\Lambda$, such that 
\[
\|u\|_{C^\alpha(B_{r/2}(x_0))} \leq C \left( \left(\fint_{B_r(x_0)} u^2 \dd x\right)^{1/2} + \textup{\mbox{Tail}}(u;x_0,r/2)\right),
\]
for any minimizer $u$ of $\mathcal J_K$ over $\mathcal H^s_g(\Omega)$, and any $B_r(x_0)\subset \Omega$ with $r<r_0$.
\end{theorem}
We emphasize that Theorem \ref{t:reg} holds for both the one-phase and two-phase problems.

Since minimizers satisfy $\mathcal L_K u = 0$ in $\{u>0\}$, the best interior regularity one can expect is H\"older continuity with an exponent depending on $\lambda$ and $\Lambda$ that could in general be very close to zero. However, for the one-phase problem, we can prove a stronger regularity result that holds at the free boundary, namely H\"older continuity of order $s$. This result is optimal in 
view of our nondegeneracy Theorem \ref{t:nondeg}. We also note that the $s$-H\"older norm is scaling-critical for the functional $\mathcal J_K$ (see Lemma \ref{l:scaling} below). Our optimal regularity result is as follows:

\begin{theorem}\label{t:s-reg}
There exist $C, r_0>0$, depending only on $d$, $s$, $\lambda$, and $\Lambda$, so that for any one-phase minimizer $u\geq 0$ of $\mathcal J_K$, and any $x_0 \in \partial \{u>0\}$, there holds
\[
\frac{u(x)}{|x-x_0|^s} \leq C \left( \left(\fint_{B_r(x_0)} u^2 \dd x\right)^{1/2} + \textup{\mbox{Tail}}(u;x_0,r/2)\right), \quad x \in B_{r/2}(x_0),
\]
whenever $B_r(x_0)\subset \Omega$ and $r<r_0$.
\end{theorem}

Next, we have a nondegeneracy result that says $u(x)$ grows like $|x-x_0|^s$ at the free boundary:
\begin{theorem}\label{t:nondeg}
    There exists a constant $C>0$ depending only on $d$, $s$, and $\Lambda$, such that for any minimizer $u$ of \eqref{e:JK} over $\mathcal H^s_{g}(\Omega)$, any free boundary point $x_0 \in \partial \{u>0\}\cap \Omega$, and any $x$ with $u(x)>0$ and $\dist(x,\partial\{u>0\}) = |x-x_0|$, there holds
    \[
    u(x) \geq C |x-x_0|^s.
    \]
\end{theorem}

The H\"older estimate of Theorem \ref{t:s-reg} and the nondegeneracy estimate of Theorem \ref{t:nondeg} imply that $u$ satisfies matching upper and lower bounds near the free boundary of the form
\begin{equation}\label{e:matching}
 c \dist(x,\partial \{u>0\})^s \leq u(x) \leq C \dist(x,\partial \{u>0\})^s.
\end{equation}
By standard arguments (see \cite{ACF1984quasilinear}), these bounds imply the following result:
\begin{corollary}\label{c:density}
    With $u$ a one-phase minimizer of $\mathcal L_K$, and $x_0\in \partial\{u>0\}$ a free boundary point, there exist a radius $r_1>0$ and a constant $c_1>0$, depending only on $d$, $s$, $\lambda$, and $\Lambda$, with
    \[
    \frac{|B_r(x_0)\cap \{u=0\}|}{|B_r(x_0)|} \geq c_1, \quad 0< r\leq r_1.
    \]
    Furthermore, there exist $r_2, c_2>0$ depending on $d$, $s$, $\lambda$, and $\Lambda$, with
    \[
    \frac{|B_r(x_0)\cap \{u>0\}|}{|B_r(x_0)|^{\min\{s/\alpha,1\}}} \geq c_2, \quad 0< r\leq r_2,
    \]
    where $\alpha$ is the H\"older exponent from Theorem \ref{t:reg}.
\end{corollary}

The exponent $\min\{s/\alpha,1\}$ appears in the lower density estimate of $\{u>0\}$ because the proof requires interior H\"older regularity in $\{u>0\}$. If it happens that $\alpha \geq s$, then one has
\[
c \leq    \frac{|B_r(x_0)\cap \{u>0\}|}{|B_r(x_0)|} \leq 1-c,       
\]
for $r$ close to zero, which also implies $\partial\{u>0\}$ has Lebesgue measure 0. It is also known that these matching regularity and nondegeneracy estimates imply the Hausdorff measure $\mathcal H^{n-\eps}(\partial \{u>0\}) = 0$ for some universal constant $\eps>0$. Again, this requires $\alpha \geq s$, which is the case, for example, if $K(x,y)$ satisfies some extra properties which imply Lipschitz continuity of solutions to the homogeneous equation $\mathcal L_K u= 0$. 

These results open the door to further investigation of the regularity of the free boundary, as established for the fractional Laplacian case in \cite{DeSR, DeS-S, EKPSS, FR}. Free boundary regularity in the most general case, without any regularity assumptions for $K(x,y)$, may be intractable with current techniques, but the case of variable kernels with some regularity in $x$ and $y$ should be in reach, and we intend to explore this in future work. 

\subsection{Proof ideas}

The proof of Theorem \ref{t:reg} is based on approximating the minimization problem for $\mathcal J_K$ with a minimization problem for the fractional, anisotropic Dirichlet energy $\mathcal E_K(w) = \iint K(x,y) |w(y) - w(x)|^2 \dd y \dd x$, whose minimizers solve $\mathcal L_K w = 0$ and are therefore H\"older continuous. In general, there is no reason that $\mathcal J_K$ and $\mathcal E_K$ are quantitatively close, but by exploiting the scaling balance between the terms in $\mathcal J_K$, we can show that suitable rescalings of the minimizer $u$ will minimize a functional similar to $\mathcal J_K$ with a small constant $\rho$ multiplying the volume penalization term. 
The idea is then to show that, if $\rho$ is sufficiently small, the minimizer of the rescaled functional is close (in an $L^2$ sense) to a minimizer $w$ of $\mathcal E_{\tilde K}$, where $\tilde K$ is a rescaled kernel that lies in the same ellipticity class as $K$. Because of this $L^2$-closeness, our minimizer $u$ inherits some regularity (in an average sense) from $w$, and this average regularity is iterated at progressively smaller scales to obtain a H\"older estimate for $u$.

We previously applied a similar strategy to a local free boundary problem in \cite{ST2022unbounded}, but we want to emphasize that the nonlocal case considered in the current article is much more difficult. The reason is that the local H\"older estimates for solutions of $\mathcal L_K w = 0$ depend on the tail of $w$, and this dependence on the tail is unavoidably transferred to $u$, the solution of our free boundary problem. This means that we need to prove quantitative control of the tail of $u$ as part of our iteration, at the same time that we control the oscillation of $u$. 

To establish the optimal regularity in the one-phase case, we need to thoroughly investigate the rigidity of the tangential problem. Specifically, we show that under suitable rescalings, minimizers are close to the ground solution (i.e. the constant function zero) rather than merely $\mathcal L_K$-harmonic functions. Consequently, additional regularity can be extracted from the tangential configuration. We highlight once more that  in the nonlocal case, investigated in this paper, this approach requires significantly more effort than in the local case.

\subsection{Related work}\label{s:related}

The modern analysis of Bernoulli-type free boundary problems represents a pivotal advancement in mathematical research, its origins rooted in the groundbreaking work of Alt and Caffarelli \cite{AC} within the one-phase scenario. The two-phase scenario, i.e. when minimizers are allowed to change sign, was investigated by Alt, Caffarelli, and Friedman \cite{ACF}. These results were followed by a large (and still active) body of work by many authors, which we will not attempt to summarize here.

Over the past two decades, there has been a notable increase in research activity surrounding the analysis of free boundary problems governed by nonlocal operators. This upsurge is driven by the vital role these problems play in numerous physical models, emphasizing their profound significance in modern scientific exploration. The fractional obstacle problem is the theme of Luis Silvestre's Ph.D. thesis \cite{Luisito1}, see also \cite{Luisito2}. Its sharp regularity was ultimately uncovered by Caffarelli, Silvestre, and Salsa in \cite{CSS}.
The investigation into the one-phase fractional Bernoulli problem began with the pioneering research of Caffarelli, Roquejoffre, and Sire \cite{CRS2010}. Subsequent studies by De Silva and Roquejoffre \cite{DeSR} delved deeper into the regularity properties of the free boundary, while an analysis of the thin version of the problem was conducted in \cite{DeS-S}, in the case $s=\frac 1 2$. More recently, \cite{EKPSS} established detailed regularity properties of the free boundary for any $s\in (0,1)$, as a consequence of corresponding results for the thin one-phase problem, and \cite{FR} investigated the question of stable minimal cones for the fractional and thin problems.

For the historical evolution of Bernoulli-type free boundary problems, we recommend the insightful and comprehensive account provided in the charming article \cite{DeS-Notices}.

While finalizing this article, we came across an interesting recent preprint by Ros-Oton and Weidner, \cite{rosoton2024}. Their primary objective is to establish boundary regularity for nonlocal elliptic operators. Notably, their Theorem 1.5 provides regularity and nondegeneracy estimates for a one-phase free boundary problem similar to ours, but with nonlocal operators whose kernels $K(h)$ satisfy $K(h) = K(-h)$ and $\lambda |h|^{-d-2s} \leq K(h) \leq \Lambda |h|^{-d-2s}$. It is worth mentioning that our kernels are more general compared to those considered in \cite{rosoton2024}, owing to their dependence on both $x$ and $y$.

As far as we are aware, other than the current article and \cite{rosoton2024}, there are no other regularity results for nonlocal Bernoulli-type problems with kernels that are more general than the fractional Laplacian. In particular, our results are new even in the case where $K(x,y) |x-y|^{d+2s}$ is smooth in $x$ and $y$.

\subsection{Organization of the paper}

In Section \ref{s:pre}, we discuss some fundamental properties of minimizers of the energy functional \eqref{e:JK}, shedding light on its critical scaling feature. Additionally, we prove that minimizers of $\mathcal J_K$ are weak subsolutions of the corresponding nonlocal operator and gather some known results regarding solutions of the homogeneous equation $\mathcal{L}_K u = 0$.

Moving on to Section \ref{s:Holder}, we establish the H\"older continuity of minimizers in the two-phase problem.

In Section \ref{s:Improved}, we focus on establishing the sharp regularity along free boundary points for nonnegative minimizers of \eqref{e:JK}.

Lastly, Section \ref{s:nond} is dedicated to obtaining the matching nondegeneracy estimate. 

\section{Preliminary results}\label{s:pre}

\subsection{Basic properties of minimizers}

Existence of minimizers of $\mathcal J_K$  follows from standard arguments, which we omit.

%
%
%

Let us examine the behavior of our minimization problem under rescalings centered around a point $x_0\in \Omega$. For $\rho, \xi\in \R$, define the more general functional 
\begin{equation}\label{e:JKrho}
\mathcal J_{K,\rho,\xi}(u) := \iint_{\R^{2d}\setminus (B_1^c)^2} K(x,y)|u(y) - u(x)|^2 \dd y \dd x +  \rho \mathcal L^d\left(\{u>\xi\}\cap B_1\right).
\end{equation}
Clearly, $\mathcal J_{K} = \mathcal J_{K,1,0}$. 

\begin{lemma}\label{l:scaling}
For some domain $\Omega$ and exterior data $g \in L^2(\R^d\setminus \Omega)$, let $u$ be a minimizer of $\mathcal J_{K,\rho,0}$ over $\mathcal H^s_g(\Omega)$. 
Then, for any $x_0\in B_1$, $\kappa \geq 0$, $\xi\in \R$, and $0<r<\mathrm{dist}(x_0,\partial \Omega)$, 
the function 
\[
v(x):= \kappa [u(x_0+rx)-\xi]
\] 
is a minimizer of $\mathcal J_{\tilde K, \tilde \rho, \xi}$ over 
\[
\mathcal H_{\tilde g}^s(B_1) = \{ w \in L^2(\R^d) \cap H^s(B_1): w(x) =\tilde g(x) \text{ a.e. in } \R^d \setminus B_1\},
\]
with
\[ 
\begin{split}
\tilde K(x,y) &= r^{d+2s} K(x_0+rx, x_0+ry),\\
\tilde \rho &= \kappa^{2} r^{2s} \rho,\\
\tilde g(x) &= \kappa [u(x_0 + rx) - \xi], \, \,\, x\in \R^d \setminus B_1.
\end{split}
\]
\end{lemma}
\eqref{e:ellipticity}.
\begin{proof}
Direct calculation. 
\end{proof}
Note that $\tilde K(x,y)$ in Lemma \ref{l:scaling} satisfies the ellipticity condition \eqref{e:ellipticity} with the same constants as $K(x,y)$.

Next, we establish the useful fact that minimizers of $\mathcal J_K$ are subsolutions of a nonlocal elliptic equation:

\begin{lemma}\label{l:subsolution}
    Let $u$ be a minimizer of $\mathcal J_{K,\rho,\xi}$ over $\mathcal H^s_g(\Omega)$ for some kernel $K(x,y)$, some constants $\rho> 0$ and $\xi\in \R$, and some $g\in L^2(\R^d\setminus \Omega)$.     Then $u$ is a weak subsolution of the integro-differential equation $\mathcal L_K u=0$ in $\Omega$. In other words, 
    \begin{equation}\label{e:subsolution-ineq}
    \iint_{\R^{2d}} K(x,y) (\phi(y) - \phi(x)) (u(y) - u(x)) \dd y \dd x \leq 0,
    \end{equation}
    for all nonnegative $\phi\in H_0^s(\Omega)$.
\end{lemma}

\begin{proof}
    This lemma follows from a standard argument. By density, we can assume $\phi$ is in $C^\infty_0(\Omega)$. We then note that $\mathcal J_{K,\rho,\xi}(u) \leq \mathcal J_{K,\rho,\xi}(u-\eps\phi)$ and $|\{u-\eps \phi>\xi\}\cap\Omega| \leq |\{u > \xi\}\cap \Omega|$, for any $\eps>0$, which implies 
    \[
    \iint_{\R^{2d}} K(x,y) |u(y) -u(x)|^2 \dd y \dd x \leq \iint_{\R^{2d}} K(x,y) |(u-\eps \phi)(y) - (u-\eps\phi)(x)|^2 \dd y \dd x.
    \]
    Dividing this inequality by $\eps$ and sending $\eps\to 0$ yields 
    \eqref{e:subsolution-ineq} after a quick calculation. 
\end{proof}

Finally, let us note that, as expected, minimizers of $u$ satisfy $\mathcal L_K u = 0$ in the weak sense at any interior point of $\{u>0\}$. This follows by comparing $\mathcal J_K(u)$ to $\mathcal J_K(u-\eps\phi)$ for a test function $\phi$  supported in $\{u>0\}$ and sending $\eps\to 0$.

\subsection{Regularity estimates for elliptic integro-differential equations}

%

Let us recall two local estimates for the equation $\mathcal L_K u=0$ from the work of Di Castro-Kuusi-Palatucci \cite{DKP2016}, specialized to $p=2$. (See also \cite{kassmann2009} for some related estimates.) It is important to note that these estimates do not require any regularity of the kernel $K(x,y)$. 

First, we have a local $L^\infty$ estimate for subsolutions:

\begin{proposition}{\cite[Theorem 1.1]{DKP2016}}\label{p:Linfty}
Let $u$ be a weak subsolution of $\mathcal L_K u = 0$ in $\Omega$, and assume that $B_R(x_0) \subset \Omega$. Then for any $\delta>0$, there holds
\[
\sup_{B_{R/2}(x_0)} u \leq \delta \, \textup{\mbox{Tail}}(u;x_0,R/2) + C_L\delta^{-d/(4s)} \left(\fint_{B_R(x_0)} u^2 \dd x \right)^{1/2},
\]
for a constant $C_L>0$ depending only on $d$, $s$, $\lambda$, and $\Lambda$.
\end{proposition}

Recall \eqref{e:tail} for the definition of the tail of $u$.

We will also need a $C^\alpha$ estimate for solutions of $\mathcal L_K u= 0$:

%

\begin{proposition}{\cite[Theorem 1.2]{DKP2016}}\label{p:holder}
    Let $u$ solve $\mathcal L_K u = 0$ in $\Omega$, and suppose that $B_{2R}(x_0) \subset \Omega$. Then for any $r\in (0,R)$, there holds
    \[
    |u(x) - u(x_0)| \leq C\left(\frac r R\right)^\beta \left[ \textup{\mbox{Tail}}(u;x_0,R) + \left(\fint_{B_{2R}(x_0)} u^2 \dd x \right)^{1/2}\right],\quad x\in B_r(x_0).
    \] 
    The constants $C>0$ and $\beta\in (0,2s)$ depend only on $d$, $s$, $\lambda$, and $\Lambda$.
\end{proposition}

\section{H\"older regularity of minimizers in the two-phase problem}\label{s:Holder}

This section contains the proof of H\"older continuity (Theorem \ref{t:reg}). The idea is to show that, at suitably small scales, $u$ is close to an $\mathcal L_K$-harmonic function $h$, and therefore inherits some regularity from $h$. The first step is the $L^2$ approximation lemma:

\begin{lemma}\label{l:L2}
Let $u$ be a minimizer of $\mathcal J_{K,\rho,\xi}$ over $\mathcal H^s_g(B_1)$, for some $K$ satisfying \eqref{e:ellipticity}, some $\rho>0$ and some $\xi\in \R$. 
Then there holds
\[
\fint_{B_{1/2}} |u - h|^2 \dd x \leq C \rho,
\]
where $h$ is the $\mathcal L_{K}$-harmonic lifting of $u$ in $B_{1/2}$, i.e. 
\[
\begin{cases}
    \mathcal L_{K} h =0, &x \in B_{1/2},\\
    h = u, &x \in \R^d\setminus B_{1/2},
\end{cases}
\]
The constant $C$ depends only on $d$, $s$, and $\lambda$.
\end{lemma}

\begin{proof}
Since $h=u$ in $\R^d\setminus B_{1/2}$, we can apply the fractional Poincar\'e inequality on $\R^d$ \cite{fractional-poincare-review} to obtain 
\[
\begin{split}
  \int_{B_{1/2}}|u-h|^2 \dd x
  &= \int_{\R^d} |u-h|^2 \dd x \\
  &\leq C \iint_{\R^{2d}} \frac {|(u-h)(x) - (u-h)(y)|^2} {|x-y|^{d+2s}} \dd x \dd y\\
  & \leq \frac C {\lambda (1-s)}\iint_{\R^{2d}}  K(x,y) |(u-h)(x) - (u-h)(y)|^2 \dd x \dd y.
  \end{split}
\]
Note that this integral over $\R^{2d}$ may equivalently be written over $\R^{2d}\setminus (B_1^c)^2$, since $u-h = 0$ outside $B_1$.

Next, since $h$ satisfies $\mathcal L_{K} h = 0$ in $B_{1/2}$, and $u-h = 0$ outside $B_{1/2}$, we have
\[
\iint_{\R^{2d}}  K(x,y) (h(y) - h(x))[(u-h)(y) - (u-h)(x)] \dd y \dd x= - 2\int_{\R^d} \mathcal L_{K}h(x) (u-h)(x) \dd x = 0,
\]
which gives
\[
  \begin{split}
  \int_{B_{1/2}}|u-h|^2 \dd x
  &\leq \frac C {\lambda (1-s)} \left( \iint_{\R^{2d}\setminus (B_1^c)^2}  K(x,y) |u(y) - u(x)|^2 \dd y \dd x \right.\\
  &\qquad \left.- \iint_{\R^{2d}\setminus (B_1^c)^2}  K(x,y) |h(y) - h(x)|^2 \dd y \dd x\right).
  \end{split}
  \]
From the minimizing property of $u$ and the fact that $h = u$ outside $B_{1/2}$, we have
\begin{equation}\label{e:vh-est1}
  \begin{split}
   \int_{B_{1/2}}|u-h|^2 \dd x
   &\leq \frac C {\lambda (1-s)}\left( \rho |\{u> \xi\}\cap B_{1/2}| - \rho |\{h>\xi\}\cap B_{1/2}|\right)\\
  &\leq \frac {C }{\lambda(1-s)} 2|B_{1/2}|\rho,
\end{split}
\end{equation}
as desired. 
\end{proof}

Next, using the previous lemma, we show a local, discrete type regularity for $u$.

\begin{lemma}\label{l:first-step}
    Let $u$ and $h$ be as in Lemma \ref{l:L2}, and assume in addition that
    \[
    \fint_{B_1} u^2 \dd x \leq 1.
    \]
    Then, for any $r_0 \in (0,1/4)$, the estimate
    \[
    \fint_{B_{r_0}} |u - h(0)|^2 \dd x \leq C  r_0^{-d} \rho + Cr_0^{2\beta} \left( \textup{\mbox{Tail}}(u;0,1/2) + \rho + 1\right)^2
    \]
    holds, where $\beta\in (0,2s)$ and $C>0$ are constants depending only on $d$, $s$, $\lambda$, and $\Lambda$. In fact, $\beta$ is the H\"older exponent from Proposition \ref{p:holder}. Furthermore, the following estimates also hold for $h$:
    \[
    \fint_{B_{1/2}} h^2 \dd x \leq C(1+\rho),
    \]
    and 
    \[
    \textup{\mbox{Tail}}(h;0,1/4) \leq C\left( (1+\rho)^{1/2} + \textup{\mbox{Tail}}(u;0,1/2)\right),
    \]
    with $C$ as above.
\end{lemma}
\begin{proof}
    From Lemma \ref{l:L2}, we have
    \[
    \fint_{B_{1/2}} |u-h|^2 \dd x \leq C_1\rho,
    \]
    and in particular,
    \begin{equation*}
    \fint_{B_{1/2}} h^2 \dd x \leq 2\fint_{B_{1/2}} |u-h|^2 \dd x + 2\fint_{B_{1/2}} u^2 \dd x \leq 2^{1+d} + 2C_1\rho \leq C(1+\rho).
    \end{equation*}
    We also have
    \[
    \begin{split}
    \mbox{Tail}(h;0,1/4) 
    &= 4^{-2s}\int_{\R^d\setminus B_{1/4}} \frac {h(x)}{|x|^{d+2s}} \dd x\\
    &= 4^{-2s} \int_{B_{1/2}\setminus B_{1/4}} \frac {h(x)}{|x|^{d+2s}} \dd x + 4^{-2s}\int_{\R^d\setminus B_{1/2}} \frac {u(x)}{|x|^{d+2s}} \dd x\\
    &\leq 4^{-2s}\left( \int_{B_{1/2}} h^2 \dd x\right)^{1/2} \left(\int_{B_{1/2}\setminus B_{1/4}} |x|^{-2d-4s} \dd x\right)^{1/2} + 2^{-2s}\mbox{Tail}(u;0,1/2)\\
    &\leq 2^{d/2} C(1+\rho)^{1/2} \omega_d^{1/2} + 2^{-2s}\mbox{Tail}(u;0,1/2)
    \end{split}
    \]    
    since $h=u$ outside $B_{1/2}$. Proposition \ref{p:holder} with $R=\frac 1 4$ then gives
    \[
        \begin{split}
            |h(x) - h(0)|
            &\leq C |x|^\beta \left[ \mbox{Tail}(h;0,1/4) + \left(\fint_{B_{1/2}} h^2 \dd x\right)^{1/2}\right]\\
            &\leq C_2 |x|^\beta \left[ \mbox{Tail}(u;0,1/2) + (1 + \rho)^{1/2}\right], \quad |x|< 1/4,
        \end{split}
    \]
    for some $C_2>0$ depending only on $d$, $s$, $\lambda$, and $\Lambda$.  
    For any $r_0\in (0,1/4)$, we therefore have
    \[
    \begin{split}
        \fint_{B_{r_0}} |u(x) - h(0)|^2 \dd x &\leq 2\left(\fint_{B_{r_0}} |u(x) - h(x)|^2 \dd x + \fint_{B_{r_0}} |h(x) - h(0)|^2 \dd x\right)\\
        &\leq 2^{d+1} C_1 r_0^{-d}\rho + C_2\left( \mbox{Tail}(u;0,1/2)^2 + 1 + \rho\right),
    \end{split}
    \]
    as desired.
\end{proof}

Now we are ready to show H\"older continuity for minimizers of $\mathcal J_{K,\rho,0}$, when $\rho$ is sufficiently small. The idea is to iterate Lemma \ref{l:first-step} via an inductive rescaling, but it is necessary to keep careful track of upper bounds for the tail at each step.

\begin{lemma}\label{l:rescaled-holder}
    There exist $r_0,\rho_0>0$ depending only on $d$, $s$, $\lambda$, and $\Lambda$, such that for any minimizer $u$ of $\mathcal J_{K,\rho,\xi}$ over $\mathcal H^s_g(\Omega)$, with $\rho\leq \rho_0$, $K$ satisfying \eqref{e:ellipticity} and \eqref{e:symmetry}, $\xi\in \R$, and $g\in L^2(\R^d\setminus B_1)$,  such that 
    \[
    \fint_{B_1} u^2 \dd x \leq 1 \quad \text{ and } \quad \mbox{\textup{Tail}}(u;0,1/2) \leq 1,
    \]
    there holds
    \[
    |u(x) - u(0)| \leq C|x|^\alpha, \quad |x|< r_0,
    \]
    where $\alpha = \dfrac \beta {6+d/s}$, $\beta$ is the exponent from Lemma \ref{l:first-step}, and $C>0$ is a constant depending only on $d$, $s$, $\lambda$, and $\Lambda$.
\end{lemma}

\begin{proof}

The constants $\rho_0$ and $r_0$ will be chosen to satisfy several conditions that arise in the course of this proof, all of which depend only on $d$, $s$, $\lambda$, and $\Lambda$. 

Our goal is to prove by induction that, for some convergent sequence $\mu_n$, the following two inequalities hold:
\begin{equation}\label{e:induction}
\fint_{B_{r_0^n}} |u - \mu_n|^2 \dd x \leq r_0^{2n\alpha},
\end{equation}
and
\begin{equation}\label{e:induction-tail}
\mbox{Tail}(u - \mu_n; 0, r_0^n/2) \leq r_0^{\alpha(n-2-d/(2s))}.
\end{equation}
The base case ($n=0$) of \eqref{e:induction} and \eqref{e:induction-tail} are true by assumption, with $\mu_0 = 0$.

Now assume that \eqref{e:induction} and \eqref{e:induction-tail} hold for some $n\geq 0$, and define 
\[
v_n(x) = \frac{u(r_0^n x) - \mu_n}{ r_0^{n\alpha}}, \quad x\in \R^d.
\]
By Lemma \ref{l:scaling}, $v_n$ is a minimizer of $\mathcal J_{K_n, \rho_n,\xi_n}$, where
\[
\begin{split}
K_n(x,y) 
&=  r_0^{n(d+2s)} K(r_0^n x, r_0^n y),\\
\rho_n 
&= r_0^{2n(s-\alpha)} \rho \leq \rho,\\
\xi_n 
&= \mu_n.
\end{split}
\]
The inductive hypothesis \eqref{e:induction} implies
\[
\fint_{B_1} v_n^2 \dd x \leq 1.
\]
and the hypothesis \eqref{e:induction-tail} implies
\begin{equation}\label{e:vkTail}
\begin{split}
    \mbox{Tail}(v_n;0,1/2) 
    &= 2^{-2s} r_0^{-n\alpha}\int_{\R^d\setminus B_{1/2}} \frac {u(r_0^n x) - \mu_n } {|x|^{d+2s}} \dd x\\
    &= 2^{-2s} r_0^{n(2s-\alpha)} \int_{\R^d\setminus B_{r_0^n/2}} \frac{ u(x) - \mu_n}{|x|^{d+2s}} \dd x\\
    &= r_0^{-n\alpha} \mbox{Tail}(u-\mu_n; 0, r_0^n/2)\\
    &\leq r_0^{-n\alpha} r_0^{\alpha(n-2-d/(2s))} \\
    &= r_0^{-\alpha(2+d/(2s))}.
\end{split}
\end{equation}
Letting $h_n$ be the $\mathcal L_{K_n}$-harmonic lifting of $v_n$ in $B_{1/2}$, Lemma \ref{l:first-step} implies, with $\mu_{n+1} = h_n(0)$ and $\rho\leq 1$,
\[
\fint_{B_{r_0}} |v_n(x) - \mu_{n+1}|^2 \dd x \leq Cr_0^{-d} \rho + Cr_0^{2\beta} \left(\mbox{Tail}(v_n;0,1/2)^2 + \rho+1\right) \leq C r_0^{-d} \rho + C r_0^{2\beta - \alpha(2+d/(2s))}.
\]
    Imposing the conditions
    \[
    r_0 \leq  \left(\frac 1 {2C}\right)^{1/\beta}, \quad \rho \leq \rho_0 \leq \frac {r_0^d}{2C}, 
    \]
     we obtain
    \begin{equation}
    \fint_{B_{r_0}} |v_n(x) - \mu_{n+1}|^2 \dd x \leq r_0^{\beta - 2\alpha(2+d/(2s))} = r_0^{2\alpha},
    \end{equation}
    since $\alpha = \beta/(6+d/s)$. 
    Translating back to $u$, this gives
    \[
    \fint_{B_{r_0^{n+1}}} |u-\mu_{n+1}|^2 \dd x \leq r_0^{2(n+1)\alpha}
    \]
    and we have shown that statement \eqref{e:induction} holds with index $n+1$ and $\mu_{n+1} = \mu_n + r_0^{n\alpha} \mu_{n+1}$. From Proposition \ref{p:Linfty} with $\delta = r_0^{2\alpha}$, and the estimates for $h$ from Lemma \ref{l:first-step}, we have
    \[
    \begin{split}
    |h_n(0)| 
    &\leq \delta \,\mbox{Tail}(h_n;0,1/4) + C_L \delta^{-d/(4s)}\left(\fint_{B_{1/2}} h_n^2 \dd x\right)^{1/2} \\
    &\leq C_{d,s}r_0^{2\alpha}(1+r_0^{-\alpha(2+d/(2s))}) + C_L r_0^{-d\alpha/(2s)} 2^{2+d} \\
    &\leq C_0 r_0^{-\alpha(d/(2s))},
    \end{split}
    \]
    which implies
    \begin{equation}\label{e:mu-k}
    |\mu_n - \mu_{n+1}| \leq C_0 r_0^{\alpha(n-d/(2s))}.
    \end{equation}

Next, we show that condition \eqref{e:induction-tail} on the tail of $u-\mu_j$ is true for $j=n+1$, so that we can close the induction. Indeed, we have
\[
\begin{split}
    \mbox{Tail}(u-\mu_{n+1}; 0, r_0^{n+1}/2) 
    &= r_0^{2s (n+1)} 2^{-2s} \int_{\R^d\setminus B_{r_0^{n+1}/2}} \frac{ u- \mu_{n+1}} {|x|^{d+2s}} \dd x \\
    &= r_0^{2s(n+1)} 2^{-2s} \left[ (\mu_n - \mu_{n+1})\int_{\R^d\setminus B_{r_0^{n+1}/2}} |x|^{-d-2s} \dd x\right.\\
    &\qquad \qquad \left. + \int_{B_{r_0^n/2} \setminus B_{r_0^{n+1}/2}} \frac{u-\mu_n}{|x|^{d+2s}} \dd x + \int_{\R^d\setminus B_{r_0^{n}/2}} \frac { u-\mu_n} {|x|^{d+2s}} \dd x \right]\\
    &= r_0^{2s(n+1)} 2^{-2s} \left[ I_1 + I_2 + I_3\right].
\end{split}
\]
To estimate $I_1$, we use \eqref{e:mu-k}:
\[
|I_1| \leq C_0 r_0^{\alpha(n-d/(2s))} \omega_d r_0^{-2s(n+1)} 2^{2s} \leq \frac 1 3 2^{2s}r_0^{-2s(n+1)+ \alpha(n-1-d/(2s))}, 
\]
since 
\[
r_0^{\alpha} \leq \frac 1 {3C_0 \omega_d}.
\]
For the middle term $I_2$, since $u-\mu_n$ is a weak subsolution of $\mathcal L_{K_n} w = 0$ (by Lemma \ref{l:subsolution}), we can apply the $L^2$-to-$L^\infty$ estimate of Proposition \ref{p:Linfty} with $\delta = r_0^{2\alpha}$ and the inductive hypotheses \eqref{e:induction} and \eqref{e:induction-tail}:
\[
\begin{split}
|I_2|
&\leq \|u - \mu_n\|_{L^\infty(B_{r_0^n/2})} \int_{B_{r_0^n/2} \setminus B_{r_0^{n+1}/2}} |x|^{-d-2s} \dd x\\
&\leq \left( C_L\delta^{-d/(4s)} \left( \fint_{B_{r_0^n}} |u-\mu_n|^2 \dd x \right)^{1/2} + \delta \,\mbox{Tail}(u-\mu_n; 0, r_0^n/2) \right) \omega_d 2^{2s}  r_0^{-2s(n+1)}\\
&\leq \left( C_Lr_0^{-\alpha d/(2s)} r_0^{n\alpha} + r_0^{2\alpha} r_0^{\alpha(n - 2 - d/(4s))} \right) \omega_d 2^{2s} r_0^{-2s(n+1)}\\
&\leq  (C_L+1) \omega_d 2^{2s}r_0^{\alpha(n-d/(2s)))} r_0^{-2s(n+1)}\\
&\leq \frac 1 3 2^{2s} r_0^{-2s(n+1) + \alpha(n-1-d/(2s))},
\end{split}
\]
where we have imposed the further condition
\[
r_0^{\alpha} \leq \frac 1 {3(C_L+1)\omega_d}.
\]
For $I_3$, we use the inductive hypothesis \eqref{e:induction-tail} on the tail of $u-u_n$ again:
\[
I_3 = 2^{2s} r_0^{-2sn} \mbox{Tail}(u-u_n;0,r_0^{n}/2) \leq 2^{2s} r_0^{-2sn +\alpha(n-2-d/(2s))} \leq \frac 1 3 2^{2s} r_0^{-2s(n+1) + \alpha(n-1-d/(2s))},
\]
since $r_0^{2s-\alpha} \leq \frac 1 3$. Combining the estimates of $I_1$, $I_2$, and $I_3$, we obtain
\[
\mbox{Tail}(u-\mu_{n+1}; 0, r_0^{n+1}/2) \leq r_0^{\alpha(n - 1 - d/(2s))},
\]
as desired.

Now that we have shown \eqref{e:induction} holds for a sequence $\mu_n$ which satisfies \eqref{e:mu-k}, we have for any $m>n\geq 0$,
\[
\begin{split}
|\mu_n - \mu_m| \leq |\mu_n - \mu_{n+1}| + \cdots + |\mu_{m-1} - \mu_m| &\leq C_0 \left( r_0^{\alpha(n-d/(2s))} + \cdots + r_0^{\alpha(m-1 - d/(2s))}\right)\\
&\leq C_0 r_0^{\alpha(n-d/(2s))} (1-r_0^\alpha)^{-1},
\end{split}
\]
using \eqref{e:mu-k}. This implies $\mu_n$ is a Cauchy sequence, and with $\mu_\infty = \lim_{n\to \infty} \mu_n$, we can take $m\to \infty$ in the previous displayed equation and obtain $|u_n - \mu_\infty| \leq C_0 r_0^{n-d/(2s)} (1-r_0)^{-1}$. 
 Finally, for any $r\in (0,r_0)$, there is an $n\geq 0$ so that $r_0^{n+1} \leq r < r_0^n$, which implies
 \[
 \begin{split}
 \fint_{B_r} |u-\mu_\infty|^2 \dd x &\leq 2\fint_{B_r} |u-\mu_n|^2 \dd x + 2|\mu_n - \mu_\infty|^2 \\
&\leq 
 2\left( \frac {r_0^n}{r}\right)^d\fint_{B_{r_0^n}} |u-\mu_n|^2 \dd x + 2|\mu_n - \mu_\infty|^2 \\
 &\leq 2\left( \frac {1}{r_0}\right)^d r_0^{2n\alpha} + 2C_0^2 r_0^{2\alpha(n-d/(2s))} (1-r_0^\alpha)^{-2}\\
 &= 2\left(\frac 1 {r_0^d} + \frac {C_0^2} {r_0^{\alpha d/s} (1-r_0^\alpha)^2} \right) r_0^{2\alpha n} =: C_1 r_0^{2\alpha n},
 \end{split}
 \]
 where we have used \eqref{e:induction}. Because of our choice of $r_0$, the constant $C_1$ depends only on $d$, $s$, $\lambda$, and $\Lambda$. This inequality, which holds for any $r\in (0,r_0)$, is known to imply the H\"older continuity of $u$ at $x=0$, with exponent $\alpha$.
\end{proof}

We are now able to prove our first main result:
\begin{proof}[Proof of Theorem \ref{t:reg}]
Let $u$ be a minimizer of $\mathcal J_K$, and let $r_0, \rho_0>0$ be the constants from Lemma \ref{l:rescaled-holder}. Let us define for some fixed $x_0\in \Omega$,
\[
w(x) = \kappa u(x_0 + rx),
\]
where 
\[
r = \min\left\{ \frac 1 2 \mbox{dist}(x_0,\partial \Omega), \rho_0^{1/(2s)}\right\}, \quad \kappa = \frac 1 {\left(\fint_{B_{r}(x_0)} u^2 \dd x\right)^{1/2} + \mbox{Tail}(u;x_0,r/2)}.
\]
Then a direct calculation shows that $\fint_{B_1} w^2 \dd x \leq 1$ and $\mbox{Tail}(w;0,1/2)\leq 1$. Lemma \ref{l:scaling} implies $w$ is a minimizer of $\mathcal J_{\tilde K, \rho,0}$ with $\rho = \kappa r^{2s}\leq \rho_0$. Therefore, the hypotheses of Lemma \ref{l:rescaled-holder} are satisfied, and we obtain 
\[
|w(x) - w(0)|\leq C|x|^\alpha, \quad |x|<r_0.
\]
Translating back to $u$, this is
\[
|u(x) - u(x_0)|\leq \frac C \kappa r^{-\alpha} |x-x_0|^\alpha, \quad |x-x_0|< r_0 r.
\]
Using this estimate and a standard covering argument, we obtain the bound
\[
\sup_{x,y\in B_{r/2}(x_0)}\frac{|u(x) - u(y)|}{|x-y|^\alpha} \leq C \frac 1 \kappa \leq C\left(\left(\fint_{B_{r}(x_0)} u^2 \dd x\right)^{1/2} + \mbox{Tail}(u;x_0,r/2)\right).
\]

Finally, let us note that the $L^\infty$ part of $\|u\|_{C^\alpha}$ is bounded via Proposition \ref{p:Linfty}, since $u$ is a weak subsolution of $\mathcal L_K u = 0$.
\end{proof}

\section{Sharp H\"older exponent at free boundary points}\label{s:Improved}

For the remainder of the paper, we focus on one-phase minimizers. The purpose of the current section is to prove optimal H\"older estimates for $u$ at free boundary points (Theorem \ref{t:s-reg}). We begin with a local smallness estimate:

\begin{lemma}\label{l:local-sup}
For any $\tau>0$, there exists $\rho_0>0$, depending only on $d$, $s$, $\lambda$, $\Lambda$, and $\tau$, such that if $u$ is a one-phase minimizer of $\mathcal J_{K,\rho,0}$ over $\mathcal H^s_{g}(B_1)$ with $K$ satisfying \eqref{e:ellipticity}, $\rho \leq \rho_0$, $g\in L^2(\R^d\setminus B_1)$, and $u$ satisfying
\begin{equation}\label{e:u-cond}
  \fint_{B_1} u^2 \dd x \leq 1,
\end{equation}
as well as $u(0) = 0$, then the upper bound
\[
\sup_{B_{1/10}} u \leq \tau  +   \frac 1 {2 \cdot 10^{2s} c_{d,s}}\textup{\mbox{Tail}}(u;0,1/2)
\]
holds, where 
\begin{equation}\label{e:cds}
c_{d,s} := \frac{d \omega_d}{2s} 2^{-2s}\frac{1}{1- 10^{-s}}.
\end{equation}
\end{lemma}

The reason for the specific value of $c_{d,s}$ will be seen during the proof of Theorem \ref{t:s-reg} below.

\begin{proof}
With $\rho>0$ to be chosen later, let $u$ be as in the statement of the lemma. 
Letting  $h$ be the $\mathcal L_K$-harmonic lifting of $u$ in $B_{1/2}$, Lemma \ref{l:L2}  implies\begin{equation}\label{e:Crho}
\fint_{B_{1/2}} |u-h|^2 \dd x \leq C \rho.
\end{equation}
From Lemma \ref{l:subsolution}, we know $u$ is a subsolution of $\mathcal L_K u = 0$, so we have $u\leq h$ in $B_{1/2}$. Since $h$ satisfies the Harnack inequality \cite[Theorem 11.1]{CS2009} and the $L^2$-to-$L^\infty$ estimate of Proposition \ref{p:Linfty}, we have
\begin{equation}\label{e:sup-bound}
\sup_{B_{1/10}} u \leq \sup_{B_{1/10}} h \leq C_1 h(0) \leq \delta \mbox{Tail}(h;0,\sigma/2) + C_L \delta^{-d/(4s)}\left( \fint_{B_{\sigma}} h^2 \dd x\right)^{1/2},
\end{equation}
for some $\sigma\in (0,1/2)$ and $\delta>0$ to be chosen later. To estimate the terms on the right, we first have  \begin{equation*}
  \begin{split}
    \fint_{B_{\sigma}} h^2 \dd x &\leq 2\fint_{B_{\sigma}} |u-h|^2 \dd x + 2\fint_{B_{\sigma}} u^2 \dd x\\
    & \leq  C \rho  + 2 \fint_{B_{\sigma}} u^2 \dd x,
    \end{split}
    \end{equation*}
    using \eqref{e:Crho}. For the last term on the right, letting $T = \mbox{Tail}(u;0,1/2)$, we apply the H\"older estimate of Lemma \ref{l:rescaled-holder} to $u/(1+T)$, which minimizes $\mathcal J_{K,\rho/(1+T),0}$, yielding, since $u(0) = 0$,
    \[
    \fint_{B_{\sigma}} u^2 \dd x = \fint_{B_{\sigma}} |u(x) - u(0)|^2 \dd x \leq C^2 \rho \sigma^{2\alpha}(1 + \mbox{Tail}(u;0,1/2)^2),
    \]
    if $\sigma<  r_0$, where $C$, $r_0$, and $\alpha$ are the constants from Lemma \ref{l:rescaled-holder}. We therefore have (letting $C$ be another constant depending only on $d$, $s$, $\lambda$, and $\Lambda$), 
    \[
    \fint_{B_\sigma} h^2 \dd x \leq C \left(\rho +  \sigma^{2\alpha}\mbox{Tail}(u;0,1/2)^2\right),
    \]
    where we used $\rho\leq 1$. 
    Next, for the tail term appearing in \eqref{e:sup-bound}, we have
    \[
    \begin{split}
    \mbox{Tail}(h;0,\sigma/2) 
    &= \sigma^{2s} 2^{-2s}\int_{\R^d\setminus B_{\sigma/2}} \frac {h(x)}{|x|^{d+2s}} \dd x\\
    &= \sigma^{2s} 2^{-2s} \int_{B_{1/2}\setminus B_{\sigma/2}} \frac {h(x)}{|x|^{d+2s}} \dd x + \sigma^{2s}2^{-2s}\int_{\R^d\setminus B_{1/2}} \frac {u(x)}{|x|^{d+2s}} \dd x\\
    &\leq \sigma^{2s} 2^{-2s}\|h\|_{L^\infty(B_{1/2})} \left(\int_{B_{1/2}\setminus B_{\sigma/2}} |x|^{-d-2s} \dd x\right) + \sigma^{2s}\mbox{Tail}(u;0,1/2)
    \end{split}
    \]  
 To bound $h$ in $L^\infty(B_{1/2})$, we note that both $u$ and $u-h$ are $\mathcal L_K$-subharmonic in $B_1$, so we can apply Proposition \ref{p:Linfty} with $R=1$ and $\delta = 1$ to both functions:
 \[
 \begin{split}
 \|h\|_{L^\infty(B_{1/2})} &\leq \|u\|_{L^\infty(B_{1/2})} + \|u-h\|_{L^\infty(B_{1/2})} \\
 &\leq \mbox{Tail}(u;0,1/2) + C_L \left(\fint_{B_1} u^2 \dd x\right)^{1/2} + C_L \left(\fint_{B_1} (u-h)^2 \dd x\right)^{1/2}\\
 &\leq \mbox{Tail}(u;0,1/2) + C(1+\rho^{1/2}),
\end{split}
\]
 where we used \eqref{e:u-cond}, \eqref{e:Crho}, and the fact that $\mbox{Tail}(u-h;0,1/2) = 0$.  We therefore can estimate
   \[
   \begin{split}
       \mbox{Tail}(h;0,\sigma/2) 
    &\leq \sigma^{2s}(C+\mbox{Tail}(u; 0,1/2)).
   \end{split}
   \]   
    Returning to \eqref{e:sup-bound}, we now have
    \[
    \sup_{B_{1/10}} u \leq   C \delta^{-d/(4s)}(\rho+ \sigma^{2\alpha}\mbox{Tail}(u;0,1/2)^2)^{1/2} +  \delta \sigma^{2s}(C+\mbox{Tail}(u;0,1/2)).
    \]
    Choosing 
    \[
    \delta = \frac \tau {2C}, \quad \rho \leq \rho_0 = \frac 1 2 \left( \frac \tau {2C}\right)^{2(1+d/(4s))},
    \]
    and
    \[
    \sigma = \min\left\{r_0, \rho^{1/(2\alpha)}, \frac 1 2 (2\delta \cdot 10^{2s} c_{d,s})^{-1/(2s)}\right\},
    \]
    we obtain the conclusion of the lemma.
\end{proof}

Next, by iterating Lemma \ref{l:local-sup}, we establish the optimal H\"older continuity of order $s$ at points of the free boundary:

\begin{proof}[Proof of Theorem \ref{t:s-reg}]

Choose $\tau = \frac 1 2 10^{-s}$, and let $\rho_0$ be the corresponding constant from Lemma \ref{l:local-sup}. We define
\[
w(x) = \kappa u(x_0 +r x),
\]
where 
\[
r = \min\left\{ \frac 1 2 \mbox{dist}(x_0,\partial \Omega), \rho_0^{1/(2s)}\right\}, \quad \kappa = \frac 1 {\sup_{B_r(x_0)} u + \mbox{Tail}(u;x_0,r/2)/(10^s c_{d,s})}.
\]
Then we have $\mbox{Tail}(w;0,1/2) \leq 1$ and $\left(\fint_{B_1} w^2 \dd x\right)^{1/2} \leq \sup_{B_1} w \leq 1$, and $w$ is a minimizer of $\mathcal J_{\tilde K,\rho, 0}$ with $\rho\leq \rho_0$ and $\tilde K$ satisfying the usual ellipticity bounds \eqref{e:ellipticity}.

Assume by induction that for some integer $n\geq 0$,
\begin{equation}\label{e:sup-induction}
\sup_{B_{10^{-n}}} w \leq \frac 1 {10^{ns}},
\end{equation}
and that 
\begin{equation}\label{e:sup-induction2}
\mbox{Tail}(w;0, 10^{-n}/2) \leq c_{d,s}10^{-(n-1)s},
\end{equation}
where $c_{d,s}$ is the constant from Lemma \ref{l:local-sup}. The base case $n=0$ is true because of our definition of $w$.

Defining
\[
w_n(x) = 10^{ns} w\left(\frac x {10^{n}}\right),
\]
Lemma \ref{l:scaling} implies that $w_n$ minimizes $\mathcal J_{K_n, \rho_n, 0}$ over $\mathcal H^{s}_{g_n}(B_1)$, with 
\[
\begin{split}
K_n(x,y) &= 10^{-n(d+2s)} K(10^{-n} x, 10^{-n}y),\\
\rho_n &= \rho,\\
g_n(x) &= 10^{-ns} u(10^{-n} x), \quad x \in \R^d\setminus B_1.
\end{split}
\]
Note that $\rho_n = \rho$ because our rescaling is critical for the functional $\mathcal J_K$.  From the inductive hypotheses, we have
\[
\fint_{B_1} w_n^2 \dd x = 10^{2ns} \fint_{B_{10^{-n}}} w^2 \dd x \leq 1,
\]
and
\[
\begin{split}
\mbox{Tail}(w_n;0,1/2) &= 2^{-2s} \int_{\R^d\setminus B_{1/2}} \frac{10^{ns} w(10^{-n} x)} {|x|^{d+2s}} \dd x\\
&= 10^{-ns} 2^{-2s} \int_{\R^d\setminus B_{10^{-n}/2}} \frac{ w(x)} {|x|^{d+2s}} \dd x\\
&= 10^{ns} \mbox{Tail}(w;0,10^{-n}/2)\\
&\leq c_{d,s}10^{s}.
\end{split}
\]
We now apply Lemma \ref{l:local-sup} with $\tau = \frac 1 2 10^{-s}$ and obtain
\[
\begin{split}
\sup_{B_{1/10}} w_n 
&\leq  \frac 1 2 10^{-s}  + \frac 1 {2 \cdot 10^{2s} c_{d,s}} \mbox{Tail}(w_n;0,1/2)\\
&\leq 10^{-s}.
\end{split}
\]
Changing variables back to $w$, this gives
\[
\sup_{B_{10^{-(n+1)}}} w \leq 10^{-(n+1)s}.
\]
Our two inductive hypotheses \eqref{e:sup-induction} and \eqref{e:sup-induction2} also imply
\[
\begin{split}
\mbox{Tail}(w;0,10^{-(n+1)}/2 )
&= 
2^{-2s} 10^{-2(n+1)s} \int_{\R^d\setminus B_{10^{-(n+1)}/2}} \frac {w(x)}{|x|^{d+2s}} \dd x\\
&= 
2^{-2s}10^{-2(n+1)s} \int_{B_{10^{-n}/2}\setminus B_{10^{-(n+1)}/2}}\frac {w(x)}{|x|^{d+2s}} \dd x + 10^{-2s}\mbox{Tail}(w;0,10^{-n}/2)\\
&\leq  
2^{-2s} 10^{-2(n+1)s} 10^{-ns} \int_{B_{10^{-n}/2}\setminus B_{10^{-(n+1)}/2}}\frac {1}{|x|^{d+2s}} \dd x + c_{d,s}10^{-2s} 10^{-(n-1)s}\\
&= 
\left(\frac {d\omega_d}{2s} 2^{-2s} + 10^{-s}c_{d,s}\right)10^{-ns} \\
&= c_{d,s} 10^{-ns},
\end{split}
\]
by our definition \eqref{e:cds} of $c_{d,s}$. This allows us to close the induction.

We have shown that \eqref{e:sup-induction} holds for all $n$. Finally, for any $x\in B_{1/10}$, we can choose $n$ so that $10^{-(n+1)} < |x| < 10^{-n}$. From \eqref{e:sup-induction}, we have
\[
w(x) \leq \sup_{B_{10^{-n}}} w \leq \frac 1 {10^{ns}} \leq 10^s |x|^s,
\]
and the proof is complete after changing variables from $w$ back to $u$.
\end{proof}

\section{Non-degeneracy estimates}\label{s:nond}

In this last section, we prove the nondegeneracy estimate of Theorem \ref{t:nondeg}. 

\begin{proof}[Proof of Theorem \ref{t:nondeg}]

    Let $x_0, z\in \Omega$ be as in the statement of the theorem, i.e. $u(z)>0$, $x_0\in \partial\{u>0\}$, and  $\dist(z,\partial\{u>0\}) = |z-x_0|$.     
    Let $r = |z-x_0|$. Since $u$ satisfies the homogeneous equation $\mathcal L_K u = 0$ in $B_{3r/4}(z)$, the Harnack inequality \cite[Theorem 11.1]{CS2009} provides a constant $C_H>0$ independent of $r$ with 
    \[
    u(x) \leq C_H u(z), \quad x \in \partial B_{r/2}(z).
    \]

    Taking a smooth cutoff function $\zeta$ with $\zeta = 0$ in $B_{r/4}(z)$ and $\zeta = 1$ outside $B_{r/2}(z)$, with $|\nabla \zeta|\leq C/r$, we define the test function
    \[
    v(x) = \min\{u(x), C_H u(z) \zeta(x)\}, \quad x\in B_{r/2}(z),
    \]
    and $v=u$ outside $B_{r/2}(z)$.  (Note that $u(z)$ is a fixed constant.) Since $v=u$ on $\partial B_{r/2}(z)$ and outside of $B_{r/2}(z)$, the minimizing property $\mathcal J_K(u) \leq \mathcal J_K(v)$ gives 
    \begin{equation}\label{e:measure}
    \begin{array}{lll}
       \displaystyle|\{u>0\}\cap B_{r/2}(z) | \hspace{-0.4cm} &-& \hspace{-0.2cm} |\{v>0\}\cap B_{r/2}(z)| \le \vspace{0.2cm} \\ 
       &\leq& \hspace{-0.2cm} \displaystyle \iint_{\R^{2d}} K(x,y)\left[  |v(y) - v(x)|^2 - |u(y) - u(x)|^2\right] \dd y \dd x \vspace{0.2cm} \\
       &=&  \hspace{-0.2cm} \displaystyle \iint_{(B_{r/2}(z))^2} K(x,y)\left[  |v(y) - v(x)|^2 - |u(y) - u(x)|^2\right] \dd y \dd x \vspace{0.2cm} \\
       &+& \hspace{-0.2cm} \displaystyle 2\int_{\R^d\setminus B_{r/2}(z)} \int_{B_{r/2}(z)}  K(x,y)\left[  |v(y) - v(x)|^2 - |u(y) - u(x)|^2\right] \dd y \dd x\vspace{0.2cm} \\
       &=:& \hspace{-0.2cm} \displaystyle I_1 + I_2,
       \end{array}
       \end{equation}
       where we have used the symmetry property $K(x,y) = K(y,x)$. 
       For $I_1$, we have
       \[
       \begin{split}
     I_1   &= C_H^2 u(z)^2\iint_{(B_{r/2}(z))^2 \cap \{u\neq v\}} K(x,y) |\zeta(y) - \zeta(x)|^2 \dd y \dd x\\
        &\leq C_H^2 u(z)^2 \Lambda \|\nabla \zeta\|_{L^\infty}^2\iint_{(B_{r/2}(z))^2} |x-y|^{2-d-2s} \dd y \dd x\\
        &\leq \frac{C C_H^2 }{r^2} u(z)^2 r^{d+2-2s},
    \end{split}
    \]
    for a constant $C>0$ depending only on $d$, $s$, and $\Lambda$.    For $I_2$, since $u=v$ outside $B_{r/2}(z)$, one has
    \[
    \begin{split}
    I_2 &= 2 \int_{\R^d\setminus B_{r/2}(z)}\int_{B_{r/2}(z)} K(x,y) (v(y) - u(y))(v(y) + u(y) + 2u(x)) \dd y \dd x \leq 0,
    \end{split}
    \]
    since $v \leq u$ and $v(y)  + u(y) + 2u(x) \geq 0$.

     Next, we bound the left side of \eqref{e:measure} from below using the fact that $u>0$ in all of $B_{r/2}(z)$:
    \[
    |\{u>0\}\cap B_{r/2}(z)| - |\{v>0\}\cap B_{r/2}(z)| = |B_{r/2}(z)| - |B_{r/2}(z)\setminus B_{r/4}(z)| = c_d r^d,
    \]
    for some $c_d>0$ depending only on the dimension $d$. We conclude
    \[
    u(z)^2 \geq C r^{2s},
    \]
    as desired.
\end{proof}




\end{document}